\documentclass[runningheads]{lncse}

\usepackage{amsmath}
\usepackage{amsfonts}
\usepackage{epsfig}
\usepackage{graphics}
\usepackage{todonotes}

\begin{document}

\title{Linear iterative schemes for doubly degenerate parabolic equations}

\titlerunning{Iterative schemes for degenerate equation}

\author{ Jakub~Wiktor~Both\inst{1} \and Kundan~Kumar\inst{1} \and Jan~Martin~Nordbotten \inst{1, 2} \and Iuliu~Sorin~Pop\inst{3,1} \and Florin~Adrian~Radu\inst{1}   }

\authorrunning{Both et al.}

\institute{
University of Bergen, Department of Mathematics, Bergen, Norway, {\tt \{jakub.both, kundan.kumar, jan.nordbotten, florin.radu\}@uib.no}
\and Princeton University, Department of Civil and Environmental Engineering, Princeton, NJ, USA
\and Hasselt University, Faculty of Sciences, Diepenbeek, Belgium, {\tt sorin.pop@uhasselt.be}
}

\maketitle

\begin{abstract}
Mathematical models for flow and reactive transport in porous media often involve non-linear, degenerate parabolic equations. Their solutions have low regularity, and therefore lower order schemes are used for the numerical approximation. Here the backward Euler method is combined with a mixed finite element method scheme, which results in a stable and locally mass-conservative scheme. At the same time, at each time step one has to solve a non-linear algebraic system, for which linear iterations are needed. Finding robust and convergent ones is particularly challenging here, since both slow and fast diffusion cases are allowed.

Commonly used schemes, like Newton and Picard iterations, are defined either for non-degenerate problems, or after regularising the problem in the case of degenerate ones. Convergence is guaranteed only if the initial guess is sufficiently close to the solution, which translates into severe restrictions on the time step. Here we discuss a linear iterative scheme which builds on the $L$-scheme, and does not employ any regularisation. We prove its rigourous convergence, which is obtained for mild restrictions on the time step. Finally, we give numerical results confirming the theoretical ones, and compare the behaviour of the scheme with other schemes.
\end{abstract}

\section{Introduction}
\label{Pop:sec:1}

We consider the following non-linear, degenerate parabolic equation
\begin{equation}\label{Pop:sec:1:eq:1}
\partial_t b(u (t, {\bf x})) - \nabla \cdot \left( \nabla u(t, {\bf x}) \right)  = f(t, {\bf x}), \quad \quad t \in (0, T], {\bf x} \in \Omega,
\end{equation}
with given functions $b: {\mathbb{R}} \rightarrow {\mathbb{R}}$ and $f:  (0, T] \times \Omega  \rightarrow {\mathbb{R}}$. $\Omega$ is a bounded domain in ${\mathbb{R}}^d, d = 1, 2\,  \rm{ or } \, 3$ having Lipschitz continuous boundary $\partial \Omega$ and $T$ is the final time. Initial and boundary conditions (for simplicity the latter are assumed homogeneous Dirichlet) are completing the problem.

Equation \eqref{Pop:sec:1:eq:1} is the transformed Richards equation after applying the Kirchhoff transformation, and in the absence of gravity (see e.g. \cite{Pop:Radu2004,Pop:radu2008}) or a diffusion equation with equilibrium sorption modelled by a Freundlich isotherm (see \cite{Pop:radu2010}). Solving \eqref{Pop:sec:1:eq:1} is of interest for many applications of societal relevance, like environmental pollution, CO$_2$ storage or geothermal energy extraction.

A particular feature of  \eqref{Pop:sec:1:eq:1} is that the the problem may become degenerate, namely change its type from parabolic into elliptic or hyperbolic. One consequence of this is that the solutions typically lack regularity. Here we assume that $b(\cdot)$ is monotone increasing and H\"older continuous, which means that two types of degeneracy are allowed in \eqref{Pop:sec:1:eq:1}. The first is when the derivative of $b(\cdot)$ vanishes ({\it{fast diffusion}}) and the second when it blows up ({\it{slow diffusion}}). In particular, the vanishing of $b^\prime(\cdot)$ may occur on intervals.

Since solutions to degenerate parabolic equations have low regularity (see \cite{Pop:alt}), low order discretisation methods are well suited for the numerical approximation of the solution. Here we combine the backward Euler (BE) method for the time discretisation with the mixed finite element method (MFEM). For the rigorous convergence analysis of the method we refer to \cite{Pop:radu2008, Pop:Pop, Pop:Klausen} and the references therein. The resulting is a scheme that is both stable and locally mass-conservative. At each time step, the outcome is a non-linear algebraic system, for which linear iterative solvers are being required.

In this paper we discuss iterative solvers for the non-linear algebraic systems arising at each time step after the complete discretisation of \eqref{Pop:sec:1:eq:1}. Observe that although referring specifically to the MFEM approach, the non-linear solvers presented here can be also be applied to other spatial discretisations, like finite volumes, conformal or discontinuous Galerkin finite elements.

The literature on non-linear solvers for \eqref{Pop:sec:1:eq:1} is very extensive, but covers in particular non-degenerate problems, or the case when $b(\cdot)$ is Lipschitz continuous. We refer to \cite{Pop:putti,Pop:park} for the Newton method, and to \cite{Pop:celia} for the modified Picard method. A combination of both is discussed in \cite{Pop:lehmann}. Also, the J\"ager-Kacur scheme was introduced in \cite{Pop:jk2}. We refer to \cite{Pop:radu2006} for the analysis of the Newton, modified Picard and the J\"ager-Kacur schemes for BE/MFEM discretisations. Recently, in \cite{Pop:BrennerCances} the capillary pressure and the saturation are expressed both in terms of a new variable, by respecting the the original saturation-capillary pressure dependency. If the new variable is properly chosen, the Richards equation receives a that is more suited for the Newton method, in the sense that all nonlinearities are Lipschitz continuous. We refer to \cite{Pop:Farthing} for a review detailing on such aspects. 

The scheme analysed here builds on the $L$-scheme, a robust fixed point method which does not involve the computations of any derivatives or a regularisation step. The convergence,  proved rigorously in \cite{Pop:pop2004,Pop:slodicka,Pop:yong}, holds in the $H^1$ norm and regardless of the initial guess, but is linear. To improve this convergence, a combination between the $L$ and the Newton schemes was discussed recently in \cite{Pop:list}. By performing first a number of $L$ iterations, one obtains an approximation that is close enough to the solution. After a switch to the Newton iterations, the convergence becomes quadratic.

Compared to the literature cited above, here we adopt a more challenging setting: $b(\cdot)$ is only H\"older continuous and not necessarily strictly increasing. This situation has plenty of practical application, e.g. when van Genucten parametrization with certain parameters is used for Richards' equation, see \cite{Pop:raduthesis}. Whenever  $b^\prime(\cdot)$ is unbounded, neither Newton or Picard methods can be applied directly. The common way to overcome this is to regularise $b(\cdot)$ (see \cite{Pop:Nochetto}), e.g. to approximate it by a Lipschitz continuous function $b_\varepsilon(\cdot)$  (see e.g. \cite{Pop:radu2010,Pop:radu2011}).
Nevertheless, a regularisation will also imply a perturbation of the solution, which affects the accuracy of the method. Here, we propose an $L$-scheme for the degenerate equation \eqref{Pop:sec:1:eq:1}, which is adapted to the H\"older continuous nonlinearity. The linear of the scheme is proved rigorously, and its performance is compared with the ones of the standard $L$- and Newton schemes, applied for the regularised problems. 

The paper is organized as follows. In the next section the fully discrete variational approximation of \eqref{Pop:sec:1:eq:1} is given and the assumptions are stated. Section \ref{Pop:sec:3} is discussing different iterative schemes. First the modified $L$-scheme together with the convergence proof are given. Then the approach based on regularisation is discussed, with particular emphasis on the Newton scheme. Finally, in Section \ref{Pop:sec:5} a comprehensive comparison between the $L$-schemes and the Newton scheme are presented. The paper is concluded with final remarks.

\section{The fully discrete aproximation}
\label{Pop:sec:2}

Throughout this paper we will use common notations in the functional analysis. By $L^p(\Omega)$ we mean the $p$-integrable functions with the norm $\| f \|_p := \big(\int_\Omega f({\bf x})\, d {\bf x})^{1/p}$, whereas ${H(\mathrm{div}; \Omega)} := \{ {\bf f} \in (L^2(\Omega))^d |  \nabla \cdot {\bf f}  \in L^2(\Omega)\}$. Further, we denote by  $\langle\cdot,\cdot \rangle$ the inner product on $L^2(\Omega)$ and by $\sigma(\Omega)$ the volume of $\Omega$. Similarly, by $H^1(\Omega)$ we mean the $L^2(\Omega)$ functions having the first order weak derivatives in $L^2$.


To define the discretisation we let $\mathcal{T}_h$ be a regular decomposition of the domain $\Omega$ ($h$ is the mesh size) and  $0=t_0<t_1<...<t_N=T$, $N\in\mathbb{N}$ is a partition of the time interval $[0,T]$ with  constant time step size $\tau=t_{k+1}-t_k$, $k\geq0$. The lowest-order Raviart-Thomas elements (see e.g.  \cite{Pop:brezzi}) are used for the discretisation in space. The spaces $W_h \times V_h \subset L^2(\Omega) \times {H(\mathrm{div}; \Omega)}$ are defined as
\begin{displaymath}
\begin{array}{l}
W_h := \{ p \in L^2(\Omega) \vert \ p_{\vert T}({{\bf x}}) = p_T \in {\mathbb{R}} \text{ for all } T \in \mathcal{T}_h \}, \\
V_h := \{ {\bf q} \in {H(\mathrm{div}; \Omega)} | {\bf q}_{\vert T}({{\bf x}})= {\bf{a_T}} + b_T {\bf{x}}, {\bf{a_T}} \in {\mathbb{R}}^d, b_T \in {\mathbb{R}} \text{ for all } T \in \mathcal{T}_h \}.
\end{array}
\end{displaymath}

The lemma below (see \cite{Pop:douglas}) will be used in the proof of Theorem  \ref{Pop:sec:2:thm:1}.
\begin{lemma} \label{Pop:sec:2:lem:1} 
There exists a constant $C_{\Omega}>0$ not depending on the mesh size $h$, such that given an arbitrary $w_h \in W_h$ there exists ${\bf v}_h \in V_h$, satisfying $\nabla \cdot  {\bf v}_h = w_h$ and $\|{\bf v}_h\| \leq C_{\Omega} \|w_h\|$.
\end{lemma}

As mentioned, \eqref{Pop:sec:1:eq:1} is completed with homogeneous Dirichlet boundary conditions, and with the initial condition $u(0, {\bf x}) := u_0 ({\bf x})$, with $u_0 \in L^2(\Omega)$. furthermore, the source term is $f \in L^2(\Omega)$. We make the following assumptions on $b(\cdot)$.
\begin{itemize}
 \item[\textbf{(A1)}]  The function $b: {\mathbb{R}} \rightarrow {\mathbb{R}}$, $b(0) = 0$ is monotone increasing and H\"older continuous: there exist $L_b  > 0$ and $ \alpha \in (0,1]$ such that
\begin{equation}\label{Pop:sec:2:A1}
|  b(x) - b(y) |  \le L_b | x - y |^\alpha \quad \quad \text{for all } \quad x, y \in {\mathbb{R}}.
\end{equation}
 \end{itemize}

\begin{remark} The case $\alpha = 1$ corresponds to a Lipschitz continuous $b(\cdot)$, a case which is relatively well-understood \cite{Pop:putti,Pop:jk2,Pop:lehmann,Pop:list,Pop:park,Pop:pop2004,Pop:radu2006,Pop:slodicka,Pop:yong}.  The case $\alpha \in (0,1)$ is encountered for the Richards equation under physically relevant parameterizations (the van Genuchten curves \cite{Pop:Jan}, see Remark 1.1 in \cite{Pop:radu2008}). Also, if Freundlich rates are used for modelling reactive transport, one has $b(u) = u + \phi(u) $, with $\phi$ increasing but non-Lipschitz. Then there exists an $m \in {\mathbb{R}}$ such that $b^{\prime} \ge m > 0$, which simplifies the analysis of the iterative schemes. 
\end{remark}

\begin{remark}  Non-linear 
convection ${\bf q}(\cdot)$ can be added, however, if being Lipschitz continuous. The numerical schemes can be then easily modified to include such changes: one can deal with such non-linearities by using either the outcome at the last iteration, or by including this term in the Newton iteration, depending on the method used. 
For the ease of presentation, such cases are not considered here.
\end{remark}

In view of the lacking regularity, the solutions to \eqref{Pop:sec:1:eq:1} are weak. We refer to \cite{Pop:alt,Pop:otto} for existence and uniqueness results. Also, the equivalence between the conformal and mixed formulation, for both time continuous and time discrete problems, is being discussed in \cite{Pop:radu2008} (see also \cite{Pop:Klausen} for a multi-point flux discretization and \cite{Pop:radu2018, Pop:radu2015arhiv} for the case of a two-phase flow model). Such results provide the existence and uniqueness of a solution for the mixed formulation, and can be used for obtaining the rigorous convergence of the discretisation. Finally, for each time step, the backward Euler-MFEM discretisation of \eqref{Pop:sec:1:eq:1} reduces to a non-linear, fully discrete variational problem ($n \geq 1$).\\[0.5em]
{\bf Problem $P^n_h$ (The non-linear fully discrete problem)}. \\Let  $u^{n-1}_h \in W_h$ be given. Find $u^{n}_h \in W_h$ and ${\bf q}^{n}_h \in V_h$ such that for any $w_h \in W_h$ and ${\bf v}_h \in V_h$ there holds
\begin{eqnarray}
\langle b(u^n_h)- b(u^{n-1}_h), w_h \rangle + \tau \langle \nabla \cdot {\bf q}^{n}_h, w_h \rangle&=& \langle f, w_h \rangle, \label{Pop:sec:2:eq:1a}\\
\langle {\bf q}^{n}_h, {\bf v}_h \rangle-\langle u^{n}_h,  \nabla  \cdot {\bf v}_h \rangle &=& 0. \label{Pop:sec:2:eq:1b} 
\end{eqnarray}
Clearly, for $n = 1$, $u^0_h$ can be taken as the $L^2$ projection of the initial condition $u_0$ onto $W_h$ (see also \cite{Pop:radu2008}).

Here we assume that a solution to Problem $P^n_h$ exists and is unique. For $\alpha = 1$, i.e. when $b$ is Lipschitz continuous, Theorem \ref{Pop:sec:2:thm:1} below guarantees that the iterative scheme is $H^1$-contractive. This immediately provides the existence of a solution. For $\alpha \in (0, 1)$, the existence can be proved by using Brouwer«s fixed point theorem (see e.g. Lemma 1.4, p. 140 in \cite{Pop:Temam}). We refer to \cite{Pop:Arbogast,Pop:Chen,Pop:Cherfils,Pop:radu2018} for similar results in the context of two-phase porous media flow models. Finally, since $b$ is monotone, uniqueness can be proved by comparison.

The main challenge in solving the non-linear Problem $P^n_h$ is to construct a linearisation scheme that is converging also for the case when $b(\cdot)$ is only H\"older continuous, implying that $b^\prime(\cdot)$ may become unbounded. The scheme is discussed in the section below. Typically, iterative approaches like the Newton, (modified) Picard, or the $L$-schemes are applied to the regularised problem, with a Lipschitz continuous approximation $b_\varepsilon)$ replacing $b$ (see \cite{Pop:putti,Pop:celia,Pop:list,Pop:park,Pop:pop2004,Pop:radu2015,Pop:slodicka}). This will be detailed in Section \ref{Pop:sec:4}.

\section{A robust iterative scheme}
\label{Pop:sec:3} 
Below we define a robust iterative scheme for \eqref{Pop:sec:2:eq:1a}-\eqref{Pop:sec:2:eq:1b}, which does not involve regularisation, or computing any derivatives. We let the time step $n \geq 1$ be fixed and assume $u^{n-1}_h \in W_h$ be given. Also, let $L = \dfrac{1}{\delta}$, where $\delta > 0$ is a small parameter that will be chosen later to guarantee that the error decreases below a prescribed threshold. With $i \in {\mathbb{N}}$, $i > 0$ being the iteration index, the iteration step is introduced through
\\[0.5em]
{\bf Problem $P_h^{n,i}$ (The $L$-scheme)}. \\
Let $u^{n, i-1}_h \in W_h$ be given. Find $(u^{n,i}_h, {\bf q}^{n,i}_h) \in W_h \times V_h$ s.t. for all $w_h \in W_h$ and ${\bf v}_h \in V_h$ one has
\begin{eqnarray}
\langle L (u^{n,i}_h - u^{n,i-1}_h) + b(u^{n,i-1}_h), w_h \rangle + \tau \langle \nabla \cdot {\bf q}^{n,i}_h, w_h \rangle&=& \langle b(u^{n-1}_h), w_h \rangle, \label{Pop:sec:2:eq:2a}
\\
\langle {\bf q}^{n,i}_h, {\bf v}_h \rangle-\langle u^{n,i}_h,  \nabla  \cdot {\bf v}_h \rangle &=& 0.  \label{Pop:sec:2:eq:2b}
\end{eqnarray}
As will be seen below, the convergence is obtained without imposing restrictions on the initial guess $u^{n,0}_h \in W_h$, but a natural choice is $u^{n-1}_h$.

As for Problem $P^n_h$, the uniqueness of a solution for Problem $P_h^{n,i}$ follows by standard techniques. Specifically, assuming that Problem $P_h^{n,i}$ has two solution pairs $(u^{n,i}_{h, k}, {\bf q}^{n,i}_{h, k}) \in W_h \times V_h$ ($k =1, 2$) and with $(du_h, {\bf dq}_h))$ denoting their difference it holds
\begin{eqnarray*}
L \langle  du_h, w_h \rangle + \tau \langle \nabla \cdot {\bf dq}_h, w_h \rangle&=& 0, 
\\
\langle {\bf q}_h, {\bf v}_h \rangle-\langle du_h,  \nabla  \cdot {\bf v}_h \rangle &=& 0,  
\end{eqnarray*}
for all $w_h \in W_h$ and ${\bf v}_h \in V_h$. Taking in the above $w_h = du_h$, respectively ${\bf v}_h = \tau {\bf du}_h$, and adding the resulting gives
\begin{equation}
L \| du_h\|^2  + \tau \|{\bf q}_h\|^2 = 0,  \label{Pop:sec:2:eq:2e}
\end{equation}
which immediately implies uniqueness. Moreover, since Problem $P_h^{n,i}$ is linear, the uniqueness also implies the existence of the solution.

To show the convergence of the scheme we define the errors
$$
e_u^{n,i}  =   u^{n,i}_h - u^{n}_h, \quad \text{ and } \quad e_{\bf q}^{n,i}   =   {\bf q}^{n,i}_h -  {\bf q}^{n}_h,
$$
where $(u^{n}_h, {\bf q}^{n}_h)$ is the solution pair of Problem $P_h^n$. For proving the convergence of the errors sequence to 0 we use the following elementary results, which hold for any $a, b \geq 0$ and $p, q > 1$ s.t. $\frac 1 p + \frac 1 q = 1$
\begin{equation}\label{Pop:sec:2:eq:2c}
a (a-b) =  \frac 1 2 \left(a^2 - b^2 + (a-b)^2\right), \quad \text{ and } \quad a b \leq \frac {a^p} p + \frac {b^q} q .
\end{equation}

We let $\delta > 0$, $L = \dfrac{1}{\delta}$, $n \in {\mathbb{N}}$, $n \geq 1$ be fixed and assume $u^{n-1}_h \in W_h$ known. The main result supporting the convergence is
\begin{theorem}\label{Pop:sec:2:thm:1}
Assuming (A1) and $\alpha \in (0, 1)$, let $i \in {\mathbb{N}}$, $i \geq 1$ and $u^{n,i-1}_h \in W_h$ be given. If $(u^{n}_h, {\bf q}^{n}_h)$ and $(u^{n,i}_h, {\bf q}^{n,i}_h)$ are the solutions of Problems $P_h^{n}$ and $P_h^{n,i}$ respectively, there holds
\begin{equation}\label{Pop:sec:2:eq:3}
\| e_u^{n,i}\|^2  + \tau \delta R(\delta,\tau) \|  e_{\bf q}^{n,i} \|^2 \le R(\delta,\tau) \| e_u^{n,i-1}\|^2 + 2 C(\alpha) R(\delta,\tau) \delta^{\frac{2}{ 1 - \alpha }}.
\end{equation}
Here $R(\delta,\tau)  = \big(1 + \dfrac{ \tau \delta }{ C_{\Omega}^2}\big)^{-1}$, $C_\Omega$ being the constant in Lemma \ref{Pop:sec:2:lem:1}, and $C(\alpha) =\frac {(1 -\alpha)} 2 \big(L_b (2\alpha)^\alpha \big)^{\frac{2}{1-\alpha}} (1 + \alpha)^{-\frac{1+\alpha}{1-\alpha}}\sigma(\Omega)$.
\end{theorem}
\begin{proof}
Subtracting  \eqref{Pop:sec:2:eq:1a} and \eqref{Pop:sec:2:eq:1b}  from \eqref{Pop:sec:2:eq:2a}, respectively \eqref{Pop:sec:2:eq:2b}, one gets for all $w_h \in W_h$ and $ {\bf v}_h \in V_h$
\begin{eqnarray}
\langle L (e_u^{n,i} - e_u^{n,i-1}) +b(u^{n,i-1}_h) -b(u^n_h), w_h \rangle + \tau  \langle \nabla \cdot e_{\bf q}^{n,i}, w_h \rangle&=& 0  , \label{Pop:sec:2:eq:4a}\\
\langle e_{\bf q}^{n,i}, {\bf v}_h \rangle-\langle e_u^{n,i},  \nabla  \cdot {\bf v}_h \rangle &=& 0 . \label{Pop:sec:2:eq:4b}
\end{eqnarray}
By taking $w_h = e_u^{n,i} \in W_h$, respectively ${\bf v}_h = \tau e_{\bf q}^{n,i} \in V_h$, adding the resulting and after some algebraic calculations one gets
\begin{equation}\label{Pop:sec:2:eq:5}
\begin{array}{l}
\dfrac{L}{2} \left(\| e_u^{n,i}\|^2  +  \| e_u^{n,i} - e_u^{n,i-1}\|^2 \right)+ \langle b(u^{n,i-1}_h) - b(u^n_h),  e_u^{n,i-1} \rangle+ \tau \|  e_{\bf q}^{n,i} \|^2 \\
\qquad = \dfrac{L}{2} \| e_u^{n,i-1}\|^2 -  \langle b(u^{n,i-1}_h) - b(u^n_h), e_u^{n,i} - e_u^{n,i-1}\rangle.
\end{array}
\end{equation}
By (A1), $  \displaystyle \langle b(u^{n,i-1}_h) - b(u^n_h),  e_u^{n,i-1} \rangle\ge {L_b^{- \frac 1 \alpha}} \|  b(u^{n,i-1}_h) - b(u^n_h) \|_{\frac{1 + \alpha}{\alpha}}^{\frac{1 + \alpha}{\alpha}}$.
Using now the inequality in \eqref{Pop:sec:2:eq:2c} with $p = \frac{1 + \alpha}{\alpha}$, $q = 1 + \alpha$, $a =  \frac{ | b(u^{n,i-1}_h) - b(u^n_h)| }{L_b^{\frac{1}{1 + \alpha}}{(\frac{ 2 \alpha}{1 + \alpha})}^{\frac{\alpha}{1 + \alpha}}}$ and $b =  L_b^{\frac{1}{1 + \alpha}} (\frac{ 2 \alpha}{1 + \alpha})^{\frac{\alpha}{1 + \alpha}} | e_u^{n,i} - e_u^{n,i-1} | $ one gets
\begin{equation}\label{Pop:sec:2:eq:6}
\begin{array}{rl}
& | \langle  b(u^{n,i-1}_h) - b(u^n_h), e_u^{n,i} - e_u^{n,i-1}\rangle| \\[0.5em]
& \; \le  \displaystyle \frac 1 {2L_b^{\frac 1 \alpha}} {\|  b(u^{n,i-1}_h) - b(u^n_h) \|_{\frac{1 + \alpha}{\alpha}}^{\frac{1 + \alpha}{\alpha}}} +\frac{(2 \alpha)^\alpha L_b}{(\alpha + 1)^{(\alpha + 1)}} \| e_u^{n,i} - e_u^{n,i-1} \|_{1 + \alpha}^{1 + \alpha}.
\end{array}
\end{equation}
From \eqref{Pop:sec:2:eq:6} and \eqref{Pop:sec:2:eq:5} one obtains
$$
\begin{array}{l}
\frac{L}{2}  \left[ \| e_u^{n,i}\|^2  +  \| e_u^{n,i} - e_u^{n,i-1}\|^2 \right] + \frac{1}{2}\langle b(u^{n,i-1}_h) - b(u^n_h),  e_u^{n,i-1} \rangle+ \tau \|  e_{\bf q}^{n,i} \|^2 \\[0.5em]
\; \le  \frac{L}{2} \| e_u^{n,i-1}\|^2 +  \dfrac{(2 \alpha)^\alpha L_b }{(\alpha + 1)^{(\alpha + 1)}} \| e_u^{n,i} - e_u^{n,i-1} \|_{1 + \alpha}^{1 + \alpha}.
\end{array}
$$
Using again Young's inequality, but with $p = \frac{2}{1+\alpha}$,  $q = \frac{2}{1 - \alpha}$, $a = \| e_u^{n,i} - e_u^{n,i-1} \|^{1 + \alpha}_{1 + \alpha} (\frac{L}{1+ \alpha})^{\frac{1+\alpha}{2}} \sigma(\Omega)^{\frac{\alpha-1}{2}}$ and $b = \frac{(2\alpha)^\alpha L_b}{(\alpha + 1)^{(\alpha + 1)}} (\frac{1 + \alpha}{L})^{\frac{1 + \alpha}{2}}\sigma(\Omega)^{\frac{1-\alpha}{2}}$ gives
$$
\begin{array}{l} 
 \dfrac{2^\alpha L_b \alpha^\alpha}{(\alpha + 1)^{(\alpha + 1)}} \| e_u^{n,i} - e_u^{n,i-1} \|_{1 + \alpha}^{1 + \alpha} \\
 \qquad \le \dfrac{L}{2} \sigma(\Omega)^{\frac{\alpha-1}{ 1 + \alpha}} \| e_u^{n,i} - e_u^{n,i-1} \|^2_{1+\alpha} + C(\alpha)  L^{\frac{1 + \alpha}{\alpha-1}} \\
 \qquad \le \dfrac{L}{2}  \| e_u^{n,i} - e_u^{n,i-1} \|^2 + C(\alpha)  L^{\frac{1 + \alpha}{\alpha-1}}, 
\end{array}
$$
where
$C(\alpha)$ is defined in the text of the theorem. In the last step above we used the inequality $\| f  \|_{1+\alpha} \le \sigma(\Omega)^{\frac{1-\alpha}{2 (1 + \alpha)}} \| f  \|_{2}$, which holds true for any $f \in L^2(\Omega)$ and $\alpha \in (0, 1]$, since $\Omega$ is bounded.
Now, from the last two 
$$
\frac{L}{2}  \| e_u^{n,i}\|^2  + \dfrac{1}{2}\langle b(u^{n,i-1}_h) - b(u^n_h),  e_u^{n,i-1} \rangle+ \tau \|  e_{\bf q}^{n,i} \|^2 \le  \dfrac{L}{2} \| e_u^{n,i-1}\|^2 + C(\alpha)  L^{\frac{1 + \alpha}{\alpha-1}}.
$$
From \eqref{Pop:sec:2:eq:4b} and using Lemma \ref{Pop:sec:2:lem:1}, a Poincare type inequality $\| e_u^{n,i}\| \le C_\Omega \|  e_{\bf q}^{n,i} \| $ can be obtained. Using this in the above,  
since $L = 
{1}/{\delta}$, one obtains \eqref{Pop:sec:2:eq:3}.
\hfill \end{proof}

\begin{remark}\label{Pop:sec:2:rem:1}
Observe that since $R(\delta,\tau) < 1$ whereas $\delta$ has a positive power in the last term on the right of \eqref{Pop:sec:2:eq:3}, this theorem gives the convergence of the scheme. More precisely, for any chosen tolerance $TOL$, one can chose $\delta$ such that the term $2 C(\alpha) \delta^{\frac{2}{ 1 - \alpha }} \frac {R(\delta,\tau)}{1-R(\delta,\tau) } < \frac 1 2 TOL$. Since this is the sum of the last terms on the right in \eqref{Pop:sec:2:eq:3}, this can be seen as the total error being accumulated while iterating in one time step. On the other hand, the first term in the right is showing how the error is contracted in one iteration. Thus, choosing $i^* \in N$ large enough s.t. $ R(\delta,\tau) ^{i^*} \| e_u^{n,0}\|^2 \leq \frac 1 2 TOL $ and applying \eqref{Pop:sec:2:eq:3} successively for $i = i^*, i^*-1, \dots, 1$ one obtains that $\| e_u^{n,i}\|^2 < TOL$. Nevertheless, the convergence rate is worsened with the decrease of $\delta$, as $R(\delta,\tau)$, approaches 1 in this case. From theoretical point of view, this results in an increased number of iterations for obtaining the desired accuracy. However, this is rather a pessimistic interpretation, as in all cases studied in Section \ref{Pop:sec:5} the number of iterations remained reasonable.
\end{remark}

\begin{remark}\label{Pop:sec:2:rem:3}
If $b$ is Lipschitz continuous, the problem reduces to the one studied in \cite{Pop:list,Pop:radu2008} and therefore we omit the proof here. In this case, the iteration is a contraction, so the convergence is unconditional and for any $L$ larger or equal to the Lipschitz constant of $b$.
\end{remark}

\begin{remark}\label{Pop:sec:2:rem:2}
Observe that the convergence can be achieved without requiring that the time step $\tau$ is small enough. In fact, when calculating the ratio $\frac {R(\delta,\tau)}{1-R(\delta,\tau) }$  one sees that $\tau$ appears in the denominator, so the larger it is, the better for the convergence of the iterative scheme. Further, the term
$2 C(\alpha) \delta^{\frac{2}{ 1 - \alpha }} \frac {R(\delta,\tau)}{1-R(\delta,\tau) }$
is practically small without taking a too small $\delta$. For example, if $\alpha = 0.5$, 
the power of $\delta$ in this term becomes 3. 
Taking for example $\delta = 0.01$ (hence $L = 1000$) gives $\delta^{\frac{1+\alpha}{1-\alpha}} = 10^{-6}$. Also, the number $C(\alpha)$ is small too. In the situation above, if $L_b = 0.5$, it is of order $10^{-4}$. 
\end{remark}

\section{Iterative schemes based on regularisation}
\label{Pop:sec:4}
As follows form the above, the iterations introduced through Problem $P^{n, i}_h$ are converging also for the case of a H\"older continuous $b$ and do not involve computing any derivatives. However, the iterations only converge linearly. A natural question appears: what is the performance of the new method in comparison with the Newton or the $L$-scheme, but applied for the regularised problems.  To study this aspect we first present below these methods and discuss their convergence.

For simplicity we consider the function $b: {\mathbb{R}} \rightarrow {\mathbb{R}}$, $b(u) = (\max\{u, 0\})^\alpha$. Observe that for $b$ is not Lipschitz for arguments approaching 0 from above. For regularising it we let $\varepsilon > 0$ and consider the functions 
$b_\varepsilon: {\mathbb{R}} \rightarrow {\mathbb{R}}$,
\begin{displaymath}\label{Pop:sec:2:eq:10}
b_\varepsilon (u)  = 
\displaystyle \varepsilon^{\alpha -1} u,  
\mbox{ or } 
b_\varepsilon (u)  = 
\displaystyle (\alpha - 1) \varepsilon^{\alpha-2} u^2 + (2 - \alpha)\varepsilon^{\alpha - 1} u, 
\end{displaymath}
if $u \in (0, \varepsilon)$, whereas $b_\varepsilon(u) = b(u)$ everywhere else.
Observe that the former is linear in $(0, \varepsilon)$, whereas the latter quadratic.
Clearly, $b_\varepsilon(\cdot)$ is nondecreasing, and both $b_\varepsilon(\cdot), b^\prime_\varepsilon(\cdot)$ are Lipschitz continuous. For the linear regularisation, the Lipschitz constants are $L_{b_\varepsilon} = \varepsilon^{\alpha-1}$, respectively $L_{b^\prime_\varepsilon} = \alpha (1-\alpha) \varepsilon^{\alpha-2}$. Moreover, it holds
$
\displaystyle 0 \le b(x) - b_\varepsilon(x) \le (1-\alpha) \alpha^{\frac \alpha {1-\alpha}} \varepsilon^\alpha.
$
Similar properties can be written for the quadratic regularisation. 

As before, with given $\varepsilon >0$ and $u^{n-1}_{h, \varepsilon} \in W_h$ (observe the dependency of the solution on $\varepsilon$), and with $i \in {\mathbb{N}}$, $i > 0$ being the iteration index, the Newton iterations for Problem $P^n_h$ are defined through \\[0.5em]
{\bf Problem $NEWTON_h^{n,i}$}. \\
Let $u^{n, i-1}_{h, \varepsilon} \in W_h$ be given. Find $(u^{n,i}_{h, \varepsilon}, {\bf q}^{n,i}_{h, \varepsilon}) \in W_h \times V_h$ s. t. for all $w_h \in W_h$ and ${\bf v}_h \in V_h$
\begin{eqnarray}
\langle b_{\varepsilon}(u^{n,i-1}_{h, \varepsilon}) + b^\prime_{\varepsilon} (u^{n,i-1}_{h, \varepsilon})(u^{n,i}_{h, \varepsilon}  - u^{n,i-1}_{h, \varepsilon}), w_h \rangle \nonumber \\
\hskip 2em + \tau \langle \nabla \cdot {\bf q}^{n,i}_{h, \varepsilon}, w_h \rangle&=& \langle b_{\varepsilon}(u^{n-1}_{h, \varepsilon} ), w_h \rangle, \label{Pop:sec:2:eq:11a}\\
\langle {\bf q}^{n,i}_{h, \varepsilon}, {\bf v}_h \rangle-\langle u^{n,i}_{h, \varepsilon},  \nabla  \cdot {\bf v}_h \rangle &=& 0. \label{Pop:sec:2:eq:11b} 
\end{eqnarray}

\begin{remark}[{\bf Regularised $L$-scheme}]
\label{Pop:sec:2:rem:3} 
An $L$-scheme for the regularised problem is obtained by replacing  $b^\prime_{\varepsilon} (u^{n,i-1}_{h, \varepsilon})$ with $L \ge 0 $ in \eqref{Pop:sec:2:eq:11a}. The resulting scheme is convergent for $ L \ge 
{L_{b_\varepsilon}}/{2}$, as proved in \cite{Pop:list,Pop:pop2004,Pop:radu2006}. Moreover, the convergence holds in $H^1$ and for any initial guess, under very mild restrictions on the time step, but it is only linear. It is worth emphasising on the difference between the $L$-scheme in Section \ref{Pop:sec:3}, designed for H{\"o}lder continuous nonlinearities, and the $L$-scheme for the regularised problems. In the former case the errors at each iteration step consist of two components, one that is contracted, and another that accumulates. The choice of the $L$ parameter is driven by these two: first, the the accumulated errors should remain below a threshold $\frac 1 2 TOL$, and second the contracted ones reduces to the same threshold. For the latter the problem is regularised so that the nonlinearities become Lipschitz continuous, and then the $L$ parameter is taken as the Lipschitz constant of the regularised nonlinearity.
\end{remark}

\begin{remark}[{\bf Convergence of the regularised Newton scheme}]
\label{Pop:sec:2:rem:4}
Two issues concerning the convergence are appearing in this case. First, the solution $u_\varepsilon$ of the regularised problem should not be too far from $u$, the solution to the original problem. This means that $\varepsilon$ should be small enough. On the other hand, the advantage of the Newton scheme is its quadratic convergence. Guaranteeing it requires typically a small $\tau$ because the method is only locally convergent, so the initial guess of the iteration should not be too far from the solution and the choice at hand is the solution at the previous time step. However, $\tau$ and $\varepsilon$ are not uncorrelated, so satisfying both requirements might be quite challenging, if not impossible in certain computations. If one assumes additionally that $b^\prime \ge m > 0$, which rules out the {\it{fast diffusion}} case, the sufficient condition for convergence is to choose $\tau = O( \varepsilon^a h^{d/2})$, with $a$ depending on the H\"older exponent (see \cite{Pop:radu2006}). In the case $b^\prime \ge  0$, one can further perturb $b$ so that $b^\prime_\varepsilon$ is bounded away from 0, e.g. by taking $b^{new}_\varepsilon(u) = \varepsilon u  + b_\varepsilon(u)$ with $b_\varepsilon(u)$ given before. Then the convergence is guaranteed for similar constraints, possibly with a different exponent $a$.
\end{remark}

To summarize, the convergence of the Newton method is depending on the choice of the discretisation and regularisation parameters. Fixing two parameters, e.g. $h$ and $\varepsilon$, only a small enough $\tau$ will guarantee the convergence. Alternatively, for fixed $\tau$ and $\varepsilon$, the mesh size can not be too small, and in case of the Newton scheme divergences, refining the mesh will not help. In other words, to achieve a certain accuracy, i.e. letting $\varepsilon \searrow 0$, the convergence condition for the Newton scheme might become very restrictive.

\section{Numerical examples}
\label{Pop:sec:5}

In this section we provide numerical examples to illustrate the performance of the method. We use the example mentioned in Section \ref{Pop:sec:4},
$ b(u) = \max\{u, 0\}^\alpha$, and for $\alpha = 0.5$. The domain is the square $\Omega=(0,1)\times(0,1)$, and the time interval is $t\in (0.0,0.5]]$.
For evaluating the convergence we take the right hand side and the initial condition so that he exact solution is
\begin{equation}  \label{Pop:example-a:prescribed-solution}
 u(t,x,y) = -\frac{1}{2} + 16\,x(1-x)\,y(1-y) (t + 0.5).
\end{equation}
We choose a source term accordingly. 

For the discretisation we have considered a $32 \times 32$ mesh with different 
time steps $\tau \in \{0.05, 0.025, 0.0125 \}$, giving 10, 20, respectively 40 time steps. To differentiate between the errors brought by the discretisation itself and those related to the iterative solver we first computed a very accurate approximation of the nonlinear, fully discrete systems. Specifically, with $\Delta u^i$ and $\Delta {\bf q}^i$ denoting the difference between two iterates, the reference solution is the iteration satisfying 
$$
\|\Delta u^i\|_{L^2(\Omega)}+\|\Delta {\bf q}^i\|_{L^2(\Omega)}< 10^{-8}, \text{ and }
 \frac{\|\Delta u^i\|_{L^2(\Omega)}}{\|u^i\|_{L^2(\Omega)}} + \frac{\|\Delta {\bf q}^i\|_{L^2(\Omega)}}{ \|{\bf q}^i\|_{L^2(\Omega)}} < 10^{-8}.
$$
This solution, called below $u_h$, was computed with the $L$-type scheme in Section \ref{Pop:sec:3} to avoid additional regularisation errors. Having obtained $u_h$ we proceed by testing the three schemes discussed here, the $L$-scheme in the framework discussed in Section \ref{Pop:sec:3} (called $HL$), and the two (Newton and $L$) in Section \ref{Pop:sec:4}, involving a regularisation step.

In agreement with the result stated in Theorem \ref{Pop:sec:2:thm:1} we choose an admissible tolerance $TOL$ to be used as stopping criterion for the different iteration schemes. Specifically, if $u_h^\star$ is the numerical $\| u_h^\star - u_h\|_{L^2(\Omega)}<TOL$ where $u_h$ is the (accurate) solution from above and $u_h^\star$ are the solutions

The numerical results are for different tolerances, namely $TOL \in \{10^{-3},$ $10^{-4}, 10^{-5}\}$. For the regularisation based schemes, the problem is first regularised by taking $\varepsilon \in \{10^{-3}$, $10^{-4}, 10^{-5}\}$. For the $L$-scheme we take $L = \varepsilon^{\alpha - 1}$, the Lipschitz constant of $b_\varepsilon$. For the $HL$-scheme we take $L=\frac 1 \delta$ where $\delta$ is such that the condition in Remark \ref{Pop:sec:2:rem:1} on the accumulated error is met.

Table~\ref{Pop:table:result:regularized-newton} gives the total number of Newton iterations and the number of iterations per time step for given different tolerances $TOL$, regularisation parameters $\varepsilon$ and time step sizes $\tau$. Clearly, if convergent the Newton scheme requires the least number of iterations. Observe that the parameters $TOL$ and $\varepsilon$ should be correlated to avoid that the regularisation error becomes dominating. In other words, a smaller $TOL$ requires a smaller $\varepsilon$ for obtaining the convergence. In the same spirit, a smaller $\tau$ requires smaller $TOL$ and $\varepsilon$. For $\tau=0.0125$, it becomes almost impossible to obtain solutions within the required accuracy by using the Newton scheme, as $\varepsilon$ has to be very small and then the condition of the Jacobian becomes very high. This is evidenced by the appearance of cases where the Newton scheme did not converge, which are mentioned as $nc$. Observe that the Newton scheme fails to converge if either the regularisation parameter $\varepsilon$ is too large for the chosen tolerance $TOL$, or if $\varepsilon$ is too low, which makes the problem very badly conditioned.

 \begin{table}[h]
  \centering
  \begin{tabular}{lll|l|l}
    $TOL$ & $\varepsilon$ & Time step $\tau$ & N-iterations & per time step\\
    \hline
    1e-3 & 1e-3 & $\{0.05,0.025,0.0125\}$ & $\{17,24,47\}$ & $\{1.7,1.2,1.2\}$ \\
    1e-3 & 1e-4 & $\{0.05,0.025,0.0125\}$ & $\{16,27,nc\}$ & $\{1.6,1.3, nc\}$ \\
    1e-3 & 1e-5 & $\{0.05,0.025,0.0125\}$ & $\{16,27,nc\}$ & $\{1.6,1.3, nc\}$ \\
    \hline
    1e-4 & 1e-3 & $\{0.05,0.025,0.0125\}$ & $\{22,41,nc\}$ & $\{2.2,2.1, nc\}$ \\
    1e-4 & 1e-4 & $\{0.05,0.025,0.0125\}$ & $\{23,48,nc\}$ & $\{2.3,2.4, nc\}$ \\
    1e-4 & 1e-5 & $\{0.05,0.025,0.0125\}$ & $\{23,46,nc\}$ & $\{2.3,2.3, nc\}$ \\
    \hline
    1e-5 & 1e-3 & $\{0.05,0.025,0.0125\}$ & $\{nc,nc,nc\}$ & $\{nc, nc, nc\}$ \\
    1e-5 & 1e-4 & $\{0.05,0.025,0.0125\}$ & $\{31,59,nc\}$ & $\{3.1,3.0, nc\}$ \\
    1e-5 & 1e-5 & $\{0.05,0.025,0.0125\}$ & $\{31,63,nc\}$ & $\{3.1,3.2, nc\}$ 
  \end{tabular}
 \caption{\label{Pop:table:result:regularized-newton} Results for the Newton scheme. The scheme does not converge for the smallest time step and if the regularisation parameter $\varepsilon$ is not in agreement with the tolerance $TOL$. }
 \end{table}
Similar experiments have been performed for the standard $L$-scheme, which can be applied after regularising the problem. Depending on $\varepsilon$, the Lipschitz constant of $b_\varepsilon$ is $L_b=\varepsilon^{\alpha-1}$. The actual values are given in Table~\ref{Pop:table:l-eps}.

 \begin{table}[h]
 \centering
  \begin{tabular}{l|lll}
   $\varepsilon$ & $10^{-3}$ & $10^{-4}$ & $10^{-5}$ \\
   \hline
   $L$ & 16  & 50 & 159 
  \end{tabular}
  \caption{\label{Pop:table:l-eps} $L$ values for the standard $L$-scheme, obtained for different values of $\varepsilon$.}
  \end{table}
Table~\ref{Pop:table:result:regularized-standard-l-scheme} gives the convergence results. As for the Newton scheme, the tolerance, the regularisation parameter and the time step should be correlated. A smaller $TOL$ requires a smaller $\varepsilon$ for obtain convergence, otherwise the regularisation error will make it impossible to meet the convergence criterion. This is the reason why the $L$ scheme, though unconditionally convergent in theory since it is a contraction, is marked as not convergent for the case $\varepsilon = 10^{-3}$, if $TOL  = 10^{-4}$ or $10^{-5}$. Also, observe that the Lipschitz constant of $b_\varepsilon$ is proportional to $\varepsilon^{\alpha-1}$, whereas the convergence rate gets close to 1 for large $L$ values, or for small time steps $\tau$ (see \cite{Pop:pop2004}). Therefore small values for $\varepsilon$ and $\tau$, combined with the finite precision arithmetic may lead again to situations where the $L$-scheme does not converge.

In fact, this is also the explanation of the fact that the number of $L$-scheme iterations increases drastically with the decrease of the regularisation parameter. Compared to the Newton scheme, the number of $L$-iterations is much larger. On the other hand, the $L$-scheme is more robust than the Newton scheme, allowing to compute the solution for small time steps $\tau$ or for small regularisation  parameters $\varepsilon$. 

   \begin{table}[h]
  \centering
  \begin{tabular}{lll|l|l}
    $TOL$ & $\varepsilon$ & Time step $\tau$ & $L$-iterations & per time step\\
    \hline
    1e-3 & 1e-3 & $\{0.05,0.025,0.0125\}$ & $\{305, 777,  1937\}$ & $\{30.5, 38.9,  48.4\}$ \\
    1e-3 & 1e-4 & $\{0.05,0.025,0.0125\}$ & $\{969, 2491, 6209\}$ & $\{96.9, 124.6, 155.2\}$ \\
    1e-3 & 1e-5 & $\{0.05,0.025,0.0125\}$ & $\{3058,7892, 19713\}$ & $\{305.8,394.6, 492.8\}$ \\
    \hline
    1e-4 & 1e-3 & $\{0.05,0.025,0.0125\}$ & $\{479, nc, nc\}$ & $\{47.9, nc, nc\}$ \\
    1e-4 & 1e-4 & $\{0.05,0.025,0.0125\}$ & $\{1505, 4058, 10920\}$ & $\{150.5, 202.9, 273\}$ \\
    1e-4 & 1e-5 & $\{0.05,0.025,0.0125\}$ & $\{4751, 12873, 34829\}$ & $\{475.1, 643.7, 870.7\}$ \\
    \hline
    1e-5 & 1e-3 & $\{0.05,0.025,0.0125\}$ & $\{nc, nc, nc\}$ & $\{nc, nc, nc\}$ \\
    1e-5 & 1e-4 & $\{0.05,0.025,0.0125\}$ & $\{2045, 5629, nc \}$ & $\{204.5, 281.5, nc \}$ \\
    1e-5 & 1e-5 & $\{0.05,0.025,0.0125\}$ & $\{6459, 17792, 49914 \}$ & $\{645.9, 889.1, 1247.9\}$ 
  \end{tabular}
 \caption{\label{Pop:table:result:regularized-standard-l-scheme} Results for the standard $L$-scheme. The scheme does not converge if the regularisation parameter $\varepsilon$ is not in agreement with the tolerance $TOL$.}
 \end{table}
Finally we draw our attention to the $HL$-scheme, where the parameter $L$ is chosen as mentioned in Remark \ref{Pop:sec:2:rem:1}, depending on $TOL$. Since the domain is the unit square one has $C_\Omega = \sigma(\Omega) = 1$ and thus $R(\delta,\tau)=(1+\tau\delta)^{-1}$. For $\alpha = 0.5$, to reduce the accumulated errors below $\frac 1 2 TOL$ one needs to take $ \delta < \frac 3 2 (\tau  TOL)^{\frac 1 3}$, while $L=\frac 1 \delta$. The corresponding values are given in Table~\ref{Pop:table:delta-tau-relation}. Observe that the values of $L$ in this case are similar to the ones for the standard $L$ scheme, except for the smallest tolerances. Also, the $L$ values increase for smaller $TOL$ and smaller time steps $\tau$, which was not the case of the standard $L$ scheme.

\begin{table}[h]
 \centering
  \begin{tabular}{lll|l}
   $TOL$ & $\tau$ & $\delta$ & $L$\\
   \hline
   1e-3 & 0.05   & 0.055 & 19 \\
   1e-3 & 0.025  & 0.044 & 23 \\
   1e-3 & 0.0125 & 0.035 & 29 \\
   \hline
   1e-4 & 0.05   & 0.025 & 40 \\
   1e-4 & 0.025  & 0.020 & 50 \\
   1e-4 & 0.0125 & 0.016 & 62 \\
   \hline
   1e-5 & 0.05   & 0.012  & 84 \\
   1e-5 & 0.025  & 0.0094 & 106 \\
   1e-5 & 0.0125 & 0.0075 & 134 
  \end{tabular}
  \caption{\label{Pop:table:delta-tau-relation} The $L$ parameters for the $HL$-scheme, computed for different values of $TOL$. The total iteration error is guaranteed below $TOL$ (see also Remark \ref{Pop:sec:2:rem:1}).}
 \end{table}
The convergence results are given in Table~\ref{Pop:table:result:hoelder-l-scheme}. Since the $L$ parameters have similar values for both $L$-type schemes, the number of iterations in both schemes is comparable whenever the standard $L$-scheme converges. However, for the $HL$-scheme, $L$ can be chosen automatically, based on the required tolerance $TOL$ and on the time step $\tau$, which can lead to lower and hence more optimal values as the convergence rate depends on $L$. Clearly, decreasing the tolerance $TOL$ leads to an increasing $L$, which deteriorates the convergence rate. However, the $HL$-scheme converged for all combinations of parameters. 

   \begin{table}[h!]
  \centering
  \begin{tabular}{ll|l|l}
    $TOL$ & Time step $\tau$ & $HL$-iterations & per time step\\
    \hline
    1e-3 & $\{0.05,0.025,0.0125\}$ & $\{370,  1143,  3581  \}$ & $\{37.0,  57.2,  89.5\}$ \\
    1e-4 & $\{0.05,0.025,0.0125\}$ & $\{1204, 4049,  13530 \}$ & $\{120.4, 202.5,  338.2\}$ \\
    1e-5 & $\{0.05,0.025,0.0125\}$ & $\{3433, 11924, 42294 \}$ & $\{343.3, 596.2, 1057.4\}$ \\
  \end{tabular}
 \caption{\label{Pop:table:result:hoelder-l-scheme} Results for the standard $HL$-scheme. The scheme converges for all values of $TOL$ and all time steps $\tau$.}
 \end{table}
When comparing the three schemes, it becomes clear that the Newton scheme requires the least number of iterations whenever this converges. On the other hand, the Newton scheme was the one which did not converge in the most of the cases considered here, so it is least robust. Also the standard $L$-scheme displayed cases where convergence failed. Besides, both schemes are involving a regularisation step. No regularisation instead is needed for the $HL$-scheme. This scheme certainly requires more iterations than the Newton scheme, but generally less than the standard $L$-scheme. Most important, it displayed a robust behaviour, as it converged in all experiments. In fact, this convergence can be achieved for any tolerance $TOL$ and time step $\tau$. 

It is worth mentioning that the total execution time is influenced not only by the number of iterations, but also by the time required to solve the linear systems corresponding to each iteration, and by the time needed to assemble the discretisation matrices. Among all three schemes, the Newton scheme is closest to generate ill conditioned matrices, if not singular. Therefore the linear solvers can become more expensive than in the case of the $L$-type schemes. Moreover, the linear system needs to be reassembled completely every iteration, as the Jacobian depends on the current iteration. On the other hand, the $L$-type schemes are better conditioned. For the example presented above, the linear systems for both $L$ schemes are involving the discrete Laplacian and the discretisation of the identity operator multiplied by $L$. This not only generates better conditioned matrices, but these matrices remain unchanged for every iteration. In this case, a solver based on the $LU$-decomposition is an effective approach, as this decomposition needs to be performed only once. 

\section{Conclusion}\label{Pop:sec:conclusion}
We discuss iterative schemes for solving the fully discrete nonlinear systems obtained by a backward Euler - lowest order Raviart-Thomas mixed finite element discretisation of a class of degenerate parabolic problem. Appearing as models of practical relevance, the nonlinear function involved in the model must be increasing and H{\" o}lder continuous, but may remain constant over intervals.  In consequence, two kinds of degeneracy are allowed, slow and fast diffusion. This leads to fully discrete systems that have singular Jacobians, which brings difficulties in finding robust iterative solvers. 

We present here an approach inspired by the $L$-scheme, which is suited for the case of H{\"o}der continuous nonlinearities. To apply the Newton scheme or the standard $L$-scheme in such a case, one needs to regularise first the problem, i.e. to approximate the nonlinearity by a Lipschitz continuous one. This step is associated with additional errors. If highly accurate approximations of the exact, fully discrete solutions are needed, the regularisation step may be the cause of the fact that the convergence is very slow, if not impossible. 
The scheme discussed here makes no use of any regularisation. Instead, the parameter $L$ is chosen not as the Lipschitz constant of the nonlinearity, but in such a way that the error has a guaranteed decay below any chosen tolerance. We provide a rigourous proof for this decay, which also gives a practical way to choose the parameter $L$. 

We present numerical experiments where we compare the behaviour of the three schemes: Newton, standard $L$, and the $L$-variant proposed here. As resulting from these experiments, the Newton scheme requires the least number of iterations, but is also the least robust of all as there were the most cases where it did not converge. The standard $L$-scheme is more robust, at the expense of a high number of iterations. Also, convergence could not be achieved in all cases, in particular if the regularisation parameter is not in agreement with the required tolerance. The new scheme combines is improving these aspects: it shows convergence for any required tolerance, and any choice of the time step. Nevertheless, an optimisation of the choice of $L$ and possibly in combination with an optimal linear solver can make the proposed scheme an effective alternative to the traditional ones.

\section*{Acknowledgement}
\noindent
The research is partially supported by the Norwegian Research Council (NFR) through the NFR-DAAD grant 255715, the VISTA project AdaSim 6367 and the project Toppforsk 250223, by Statoil through the Akademia Grant and by the Research Foundation-Flanders (FWO) through the Odysseus programme (project G0G1316N).

\ifx\undefined\bysame
\newcommand{\bysame}{\leavevmode\hbox to3em{\hrulefill}\,}
\fi

\end{document}